\documentclass[11pt]{article}
\usepackage{etex}
\reserveinserts{28}
\usepackage[all]{xy}

\usepackage{amsfonts}
\usepackage{amsthm}
\usepackage{enumerate}
\usepackage{graphicx}
\usepackage{mathrsfs}
\usepackage{bm}
\usepackage{cite}
\usepackage{amssymb,amsmath} 
\usepackage{makecell}
\usepackage{tikz}
\usepackage{pgf}
\usepackage{tikz}
\usetikzlibrary{patterns}
\usepackage{pgffor}
\usepackage{pgfcalendar}
\usepackage{pgfpages}
\usepackage{shuffle,yfonts}
\usepackage{mathtools}
\DeclareFontFamily{U}{shuffle}{}
\DeclareFontShape{U}{shuffle}{m}{n}{ <-8>shuffle7 <8->shuffle10}{}

\newcommand{\ola}{\overleftarrow}
\newcommand{\ora}{\overrightarrow}

\newcommand{\bfk}{{\boldsymbol{\sl{k}}}}

\newcommand{\bfx}{{\boldsymbol{\sl{x}}}}

 \allowdisplaybreaks
\usetikzlibrary{arrows,shapes,chains}
 \allowdisplaybreaks

\catcode`!=11
\let\!int\int \def\int{\displaystyle\!int}
\let\!lim\lim \def\lim{\displaystyle\!lim}
\let\!sum\sum \def\sum{\displaystyle\!sum}
\let\!sup\sup \def\sup{\displaystyle\!sup}
\let\!inf\inf \def\inf{\displaystyle\!inf}
\let\!cap\cap \def\cap{\displaystyle\!cap}
\let\!max\max \def\max{\displaystyle\!max}
\let\!min\min \def\min{\displaystyle\!min}
\let\!frac\frac \def\frac{\displaystyle\!frac}
\catcode`!=12

\let\oldsection\section
\renewcommand\section{\setcounter{equation}{0}\oldsection}

\allowdisplaybreaks

\DeclareMathOperator*{\dep}{dep}
\DeclareMathOperator{\Li}{Li}

\def\R{\mathbb{R}}

\def\N{\mathbb{N}}

\def\ze{\zeta}

\theoremstyle{plain}
\newtheorem{thm}{Theorem}[section]
\newtheorem{lem}[thm]{Lemma}
\newtheorem{cor}[thm]{Corollary}
\newtheorem{con}[thm]{Conjecture}
\newtheorem{pro}[thm]{Proposition}
\theoremstyle{definition}

\newtheorem{re}[thm]{Remark}
\newtheorem{exa}[thm]{Example}

\begin{document}
\title{\bf Mneimneh-type Binomial Sums of Multiple Harmonic-type Sums}
\author{
{Ende Pan$^{a,}$\thanks{Email: 13052094150@163.com}\quad{and}\quad Ce Xu$^{b,}$\thanks{Email: cexu2020@ahnu.edu.cn}}\\[1mm]
a. \small College of Teacher Education, Quzhou University, \\ \small Quzhou 324022, P.R. China\\
b. \small School of Mathematics and Statistics, Anhui Normal University,\\ \small Wuhu 241002, P.R. China
}

\date{}
\maketitle

\noindent{\bf Abstract.} In this paper, we establish some expressions of Mneimneh-type binomial sums involving multiple harmonic-type sums in terms of finite sums of Stirling numbers, Bell numbers and some related variables. In particular, we present some new formulas of Mneimneh-type binomial sums involving generalized (alternating) harmonic numbers. Further, we establish a new identity relating the multiple zeta star values $\ze^\star(m+2,\{1\}_{r-1})$ and specific multiple polylogarithms by applying the Toeplitz principle. Furthermore, we present some interesting consequences and illustrative examples.

\medskip

\medskip


\section{Introduction}

We begin with some basic notations. Let $\N$ be the set of positive integers and $\N_0:=\N\cup \{0\}$.
A finite sequence $\bfk:=(k_1,\ldots, k_r)\in\N^r$ is called a \emph{composition}. We put
\begin{equation*}
 |\bfk|:=k_1+\cdots+k_r,\quad \dep(\bfk):=r,
\end{equation*}
and call them the weight and the depth of $\bfk$, respectively. If $k_1>1$, $\bfk$ is called \emph{admissible}. As a convention, we denote by $\{1\}_d$ the sequence of 1's with $d$ repetitions.

For a composition $\bfk=(k_1,\ldots,k_r)$ and positive integer $n$, the \emph{multiple harmonic sums} (MHSs) and \emph{multiple harmonic star sums} (MHSSs) are defined by
\begin{align}
\zeta_n(\bfk):=\sum\limits_{n\geq n_1>\cdots>n_r>0 } \frac{1}{n_1^{k_1}\cdots n_r^{k_r}}\quad
\text{and}\quad
\zeta^\star_n(\bfk):=\sum\limits_{n\geq n_1\geq\cdots\geq n_r>0} \frac{1}{n_1^{k_1}\cdots n_r^{k_r}}\label{MHSs+MHSSs},
\end{align}
respectively. If $n<r$ then ${\zeta_n}(\bfk):=0$ and ${\zeta _n}(\emptyset )={\zeta^\star _n}(\emptyset ):=1$. In particular, if $\bfk=(k)\in \N$ in \eqref{MHSs+MHSSs}, then
\begin{align}\label{defn-gener-harmonicnumber}
\ze_n(k)= \ze^\star_n(k)\equiv H_n^{(k)}=\sum_{j=1}^n \frac{1}{j^k}
\end{align}
is the $n$-th generalized \emph{harmonic number} of order $k$, and furthermore, if $k=1$ then $H_n\equiv H_n^{(1)}$ is the classical $n$-th harmonic number.
When taking the limit $n\rightarrow \infty$ in \eqref{MHSs+MHSSs} we get the so-called the \emph{multiple zeta values} (MZVs) and the \emph{multiple zeta star values} (MZSVs), respectively (\cite{H1992,DZ1994,Zhao2016})
\begin{align*}
{\zeta}( \bfk):=\lim_{n\rightarrow \infty}{\zeta _n}(\bfk) \quad
\text{and}\quad
{\zeta^\star}( \bfk):=\lim_{n\rightarrow \infty}{\zeta^\star_n}( \bfk),
\end{align*}
defined for an admissible composition  $\bfk$ to ensure convergence of the series. More generally, let $\bfk=(k_1,\ldots,k_r)\in\N^r$ and $\bfx=(x_1,\dotsc,x_r)$ where $x_1,\dotsc,x_r$ are complex variables.
The classical \emph{multiple polylogarithm} (MPL) and \emph{multiple polylogarithm star function}
with $r$ variables are defined by
\begin{align}
\Li_{\bfk}(\bfx):=\sum_{n_1>n_2>\cdots>n_r>0} \frac{x_1^{n_1}\dotsm x_r^{n_r}}{n_1^{k_1}\dotsm n_r^{k_r}}\quad
\text{and}\quad
\Li^\star_{\bfk}(\bfx):=\sum_{n_1\geq n_2\geq \cdots\geq n_r>0} \frac{x_1^{n_1}\dotsm x_r^{n_r}}{n_1^{k_1}\dotsm n_r^{k_r}},
\end{align}
respectively,
which converge if $|x_1\cdots x_{j}|<1$ for all $j=1,\dotsc,r$.
For a composition $\bfk=(k_1,\ldots,k_r)$ and positive integer $n$, and $z\in \R$, we define the following \emph{multiple harmonic-type sums}
\begin{align}\label{defn-mhtss}
\zeta^\star_n(\bfk;z):=\sum\limits_{n\geq n_1\geq\cdots\geq n_r>0} \frac{z^{n_r}}{n_1^{k_1}\cdots n_r^{k_r}}.
\end{align}

The motivation of this paper arises from the results of Mneimneh \cite{M2023}, Campbell \cite{C2023}, Komatsu-Wang \cite{KW2024} and Gen$\check{\rm c}$ev \cite{G2024}. In a 2023 \emph{Discrete Mathematics} paper, Mneimneh \cite{M2023} introduced the following remarkable formula for a binomial sum of harmonic numbers
\begin{align}\label{Mneimneh-identity}
\sum_{k=0}^n H_k\binom{n}{k} p^k(1-p)^{n-k}=\sum_{i=1}^n \frac{1-(1-p)^i}{i}\quad (p\in [0,1]).
\end{align}
The result in \eqref{Mneimneh-identity} is of interest due to how it generalizes the known elegant formula (\cite{PS2003,Spi2007,Spi2019})
\begin{align*}
\sum_{k=0}^n H_k \binom{n}{k}= 2^n \left( H_n-\sum_{j=1}^n \frac{1}{j2^j}\right).
\end{align*}
In 2023, Campbell \cite{C2023} gave two new proofs of \eqref{Mneimneh-identity} by using Zeilberger's algorithm \cite{PZ1996} and beta-type integral formula. Further, in 2024, Komatsu and Wang \cite{KW2024} extended Mneimneh's formula to the generalized hyperharmonic numbers. In particular, they proved that this formula \eqref{Mneimneh-identity} can be extended to the generalized harmonic numbers and established explicit formulas.
Gen$\check{\rm c}$ev \cite[Thm. 2.1]{G2024} established a generalization of Mneimneh summation formula
\begin{align}\label{equ-exp-z}
\sum_{k=0}^n \binom{n}{k}\ze^\star_k(r;z)p^k(1-p)^k=\sum_{n\geq n_1 \geq \cdots n_r\geq 1} \frac{(1-p)^{n_1}\left(\left(1+\frac{zp}{1-p}\right)^{n_r}-1\right)}{n_1\cdots n_r},
\end{align}
where $r\in \N,p,z\in \R$ and $p\neq 1$. Further, Gen$\check{\rm c}$ev \cite{G2024} also mentioned that another type of generalization of the Mneimneh binomial sums by replacing the generalized harmonic numbers with the multiple harmonic sums could be considered.
Very recently, the current authors \cite{PX2024} used the method of integrals of natural logarithms to establish a general Mneimneh-type binomial sums involving Bell numbers and some Mneimneh-type binomial sums involving (alternating) harmonic numbers.

In this paper, we will use the methods of our previous paper \cite{PX2024} to establish some new formulas of Mneimneh-type binomial sums involving multiple harmonic-type sums. The main results are the following four theorems.
\begin{thm} \label{thm:Mneimneh-type-mhtss} For any reals $x,y$ and $z\in (-\infty,1]$ with $x/(x+y)\geq 0$ and $n,p\in \N$, we have
\begin{align}\label{equ:Mneimneh-type-mhtss}
&\sum_{k=0}^n x^ky^{n-k}\binom{n}{k} \zeta_k^\star(\{1\}_p;z)\nonumber\\&=(-1)^{p-1} (x+y)^n \sum_{j=1}^n \left(\Big(\frac{xz+y}{x+y}\Big)^j-\Big(\frac{y}{x+y}\Big)^j\right)\left\{\sum_{i=0}^{p-1} (-1)^i \frac{Y_i(n)}{i!} \frac{s(j,p-i)}{j!}\right\}.
\end{align}
\end{thm}

\begin{thm} \label{thm:Mneimneh-type-mhtss-2} For any reals $x,y$ and $z\in (-\infty,1]$ with $x/(x+y)\geq 0$ and $n,p,m\in \N$, we have
\begin{align}\label{equ:Mneimneh-type-mhtss-2}
&\sum_{k=0}^n x^ky^{n-k}\binom{n}{k} \zeta_k^\star(\{1\}_{p-1},m+1;z)\nonumber\\&=(-1)^{m+p}(x+y)^n \sum_{n\geq j\geq l\geq 1} \left(\frac{y}{x+y}\right)^j \left( \left(1+\frac{xz}{y}\right)^l-1\right)\nonumber\\&\quad\times\left\{\sum_{i=0}^{p-1} (-1)^i \frac{Y_i(n)}{i!} \frac{s(j,p-i)}{j!}\right\}\left\{\sum_{h=0}^{m-1} (-1)^h \frac{Y_h(j)}{h!} \frac{s(l,m-h)}{l!}\right\}.
\end{align}
\end{thm}

\begin{thm} \label{thm:Mneimneh-type-mhtss-3} For any reals $x,y$ with $x/(x+y)\geq 0$ and $n,p,r\in \N,\ m\in \N_0$, we have
\begin{align}\label{equ:Mneimneh-type-mhtss-3}
&\sum_{k=0}^n x^ky^{n-k}\binom{n}{k} \zeta_k^\star(\{1\}_{p-1},m+2,\{1\}_{r-1})\nonumber\\&=(x+y)^n \sum_{n\geq n_1\geq \cdots\geq n_{p+m+r}\geq 1} \frac{\left(\frac{y}{x+y}\right)^{n_p} \left(\frac{x+y}{y}\right)^{n_{p+m+1}}\left(1-\left(\frac{y}{x+y}\right)^{n_{p+m+r}}\right)}{n_1\cdots n_{p+m+r}}.
\end{align}
\end{thm}
Here $s(n,k)$ and $Y_k(n)$ stand for the (unsigned) Stirling numbers of the first kind and Bell numbers (see Section \ref{Intr-STT-Bell}), respectively.

\begin{thm} \label{thm:Mneimneh-type-mhtss-4} For integer $m\in \N_0$ and real $y\in (0,1)$, if $r\in \N\setminus\{1\}$ then
\begin{align}\label{equ:Mneimneh-type-mhtss-4}
\ze^\star(m+2,\{1\}_{r-1})=\Li^\star_{\{1\}_{m+r+1}}\big(y,\{1\}_m,y^{-1},\{1\}_{r-1}\big)-\Li^\star_{\{1\}_{m+r+1}}\big(y,\{1\}_m,y^{-1},\{1\}_{r-2},y\big),
\end{align}
if $r=1$ then
\begin{align}\label{equ:Mneimneh-type-mhtss-4-2}
\ze^\star(m+2)=\Li^\star_{\{1\}_{m+2}}\big(y,\{1\}_m,y^{-1}\big)-\Li^\star_{\{1\}_{m+2}}\big(y,\{1\}_{m+1}\big).
\end{align}
\end{thm}

The remainder of this paper is organized as follows. In Section \ref{Intr-STT-Bell}, we briefly introduce the (unsigned) Stirling number of the first kind and Bell polynomials, and give two relations of Stirling numbers, Bell numbers and multiple harmonic (star) sums. In Section \ref{sec-main-result-proof}, we first give three lemmas and prove two propositions. Then we apply these lemmas and propositions obtained to prove Theorems \ref{thm:Mneimneh-type-mhtss}-\ref{thm:Mneimneh-type-mhtss-3}. Further, we present some interesting corollaries and illustrative cases. Finally, we prove the Theorem \ref{thm:Mneimneh-type-mhtss-4} by applying the Toeplitz principle. In Section \ref{sec-conj}, we end this paper with a conjecture of Mneimneh-type binomial sums involving multiple harmonic star sums.

\section{Preliminaries}\label{Intr-STT-Bell}

\subsection{(unsigned) Stirling number of the first kind}\label{Intr-Stirling number}
Let $s(n,k)$ denote the (unsigned) Stirling number of the first kind, which is defined by \cite{CG1996,C1974}
\begin{align}\label{SN-Dgf}
n!x\left( {1 + x} \right)\left( {1 + \frac{x}{2}} \right) \cdots \left( {1 + \frac{x}{n}} \right) = \sum\limits_{k = 0}^n s(n+1,k+1) x^{k + 1},
\end{align}
with $s(n,k):=0$ if $n<k$, and $s(n,0)=s(0,k):=0,\ s(0,0):=1$, or equivalently, by the generating function:
\begin{align*}
{\log ^k}( {1 - x} ) = {\left( { - 1} \right)^k}k!\sum\limits_{n = 1}^\infty  {s(n,k)\frac{{{x^n}}}{{n!}}}\quad (x \in [- 1,1)).
\end{align*}
The Stirling numbers ${s\left( {n,k} \right)}$ of the first kind satisfy a recurrence relation in the form
\[s( {n,k}) = s({n - 1,k - 1}) + \left( {n - 1} \right)s( {n - 1,k})\quad (n,k \in \N).\]
Obviously, ${s\left( {n,k} \right)}$ can be expressed in terms of a linear combinations of products of harmonic numbers (see \eqref{defn-gener-harmonicnumber}). In particular,
\begin{align*}
& s( {n,1} ) = \left( {n - 1} \right)!,\\& s\left( {n,2} \right) = \left( {n - 1} \right)!{H_{n - 1}},\\& s( {n,3} ) = \frac{{\left( {n - 1} \right)!}}{2}\left[ {H_{n - 1}^2 - {H^{(2)} _{n - 1}}} \right],\\
&s( {n,4} ) = \frac{{\left( {n - 1} \right)!}}{6}\left[ {H_{n - 1}^3 - 3{H_{n - 1}}{H^{(2)} _{n - 1}} + 2{H^{(3)} _{n - 1}}} \right], \\
&s( {n,5} ) = \frac{{\left( {n - 1} \right)!}}{{24}}\left[ {H_{n - 1}^4 - 6{H^{(4)} _{n - 1}} - 6H_{n - 1}^2{H^{(2)} _{n - 1}} + 3(H^{(2)} _{n - 1})^2 + 8H_{n - 1}^{}{H^{(3)} _{n - 1}}} \right].
\end{align*}

In particular, we proved the following relation (\cite[Thm 2.5]{Xu2017})
\begin{align}\label{equ-Str-mhs-1}
s\left( {n,k} \right) = \left( {n - 1} \right)!{\zeta _{n - 1}}( {{{\left\{ 1 \right\}}_{k - 1}}})\quad (k,n\in \N).
\end{align}

\subsection{Bell polynomials}\label{Intr-Bell polynomials}

Define the \emph{exponential partial Bell polynomials} $B_{n,k}$ by
\[
\frac{1}{k!}\left(\sum_{n=1}^\infty x_n\frac{t^n}{n!}\right)^k
    =\sum_{n=k}^\infty B_{n,k}(x_1,x_2,\ldots,x_n)\frac{t^n}{n!}\,,\quad k=0,1,2,\ldots\,,
\]
and the \emph{exponential complete Bell polynomials} $Y_n$ by
\[
Y_n(x_1,x_2,\ldots,x_n):=\sum_{k=0}^nB_{n,k}(x_1,x_2,\ldots,x_n)
\]
(see \cite[Section 3.3]{C1974}). According to \cite[Eq. (2.44)]{Riordan58}, the complete Bell polynomials $Y_n$ satisfy the recurrence
\[
Y_0=1\,,\quad
Y_n(x_1,x_2,\ldots,x_n)=\sum_{j=0}^{n-1}\binom{n-1}{j}x_{n-j}Y_j(x_1,x_2,\ldots,x_j)
    \,,\quad n\geq1\,,
\]
from which, the first few polynomials can be obtained immediately:
\begin{align*}
&Y_0=1\,,\quad
    Y_1(x_1)=x_1\,,\quad
    Y_2(x_1,x_2)=x^2_1+x_2\,,\quad
    Y_3(x_1,x_2,x_3)=x^3_1+3x_1x_2+x_3\,,\\
&Y_4(x_1,x_2,x_3,x_4)=x^4_1+6x^2_1x_2+4x_1x_3+3x^2_2+x_4\,.
\end{align*}

Define the \emph{Bell number} ${Y_k}( n )$ by
\begin{align}
{Y_k}( n ):= {Y_k}( {{H_n},1!{H_n^{(2)}},2!{H_n^{(3)}}, \cdots ,( {k - 1})!{H_n^{(k)}}} ).
\end{align}
Clearly, the Bell number ${Y_k}( n )$ is a rational linear combination of products of harmonic numbers. We have
\begin{align*}
&{Y_1}( n ) = {H_n},\\&{Y_2}( n ) = H_n^2 + {H^{(2)} _n},\\&{Y_3}( n ) =  H_n^3+ 3{H_n}{H^{(2)} _n}+ 2{H^{(3)} _n},\\
&{Y_4}( n ) = H_n^4 + 8{H_n}{H^{(3)} _n} + 6H_n^2{H^{(2)} _n} + 3(H^{(2)} _n)^2 + 6{H^{(4)} _n},\\
&{Y_5}( n ) = H_n^5 + 10H_n^3{H^{(2)} _n} + 20H_n^2{H^{(3)}_n} + 15{H_n}({H^{(2)}_n})^2 + 30{H_n}{H^{(4)} _n}+ 20{H^{(2)} _n}{H^{(3)} _n} + 24{H^{(5)} _n}.
\end{align*}

In \cite[Eq. (2.9)]{Xu2017}, we showed
\begin{align}\label{equ-Str-mhss-1}
\zeta _n^ \star \left( {{{\{ 1\} }_m}} \right) = \frac{1}{{m!}}{Y_m}\left( n \right)\quad (n,m\in \N_0).
\end{align}

\section{Proofs of Main Results}\label{sec-main-result-proof}

\subsection{Lemmas and Propositions}

Next, we present some lemmas and propositions, which are useful in the development of our main theorems.

\begin{lem}(\cite[Thm. 2.9]{XYS2016} )\label{lem-Bell-Log} For $n\in \N$ and $p\in \N_0$, we have
\begin{align}\label{equ-Bell-Log}
\int\limits_0^1 {{t^{n - 1}}{{\log}^p}(1 - t)} dt = {( { - 1} )^p}p!\frac{\zeta_n^\star(\{1\}_p)}{n}.
\end{align}
\end{lem}

\begin{lem}(\cite[Thm. 2.2]{Xu2017}) \label{lem-x-mhs-Log} For $n,p\in \N$ and $x\in(-\infty,1)$, we have
\begin{align}\label{equ-log-x-m}
&\int\limits_0^x {{t^{n - 1}}{{\log}^p}\left( {1 - t} \right)} dt =p!\frac{{{{\left( { - 1} \right)}^p}}}{n}\zeta _n^ \star ( {{{\{ 1\} }_p};x} )\nonumber\\
&\quad\quad\quad\quad\quad\quad\quad\quad\quad+\frac{1}{n}\sum\limits_{j =0}^{p- 1} {{{\left( { - 1} \right)}^{j}}j!\binom{p}{j}{{\log}^{p - j}}\left( {1 - x} \right)} \left( {\zeta _n^ \star( {{{\{ 1\} }_j};x}) - \zeta _n^ \star( {{{\{ 1\} }_j}} )} \right).
\end{align}
(Note that in \cite[Thm. 2.2]{Xu2017}, the range of values for $x$ is $x\in[-1,1)$, but in fact, the above equation also holds for $x\in (-\infty,-1]$.)
\end{lem}

\begin{lem}(\cite[Thm. 4.3]{XuZhao2020d})\label{lem-PMPLs2}
For ${\bfk}=(k_1,\ldots,k_r)\in \N^r$ and $l\in\N_0,n\in\N$, we have
\begin{align}\label{eq-PMPLs2}
&\sum_{n\geq n_1\geq n_2\geq \cdots\geq n_r>0}  \prod_{j=1}^r \frac{1}{(n_{j}+l)^{k_{j}}}
=(-1)^r \sum_{j=0}^r (-1)^j\ze^\star_{n+l}(\ora\bfk_{\hskip-2pt j}) \ze_l(\ola\bfk_{\hskip-2pt j+1}),
\end{align}
where $\ora\bfk_{\hskip-2pt j}:=(k_1,\ldots,k_j)$ and $\ola\bfk_{\hskip-2pt j}:=(k_r,\ldots,k_j)$ for all $1\le j\le r$.
\end{lem}

\begin{pro}\label{pro-one-miter} For reals $a,z$ with $a\neq 0$ and positive integers $m,j$, we have
\begin{align}\label{equ-mi-fs-za}
&\int_{z>t_1>\cdots>t_m>0} \frac{(at_m+1)^j-1}{t_1\cdots t_m}dt_m\cdots dt_1\nonumber\\
&=(-1)^{m-1} \sum_{l=1}^j ((az+1)^l-1)\left\{\sum_{h=0}^{m-1} (-1)^h \frac{Y_h(j)}{h!} \frac{s(l,m-h)}{l!}\right\}.
\end{align}
\end{pro}
\begin{proof}
By a direct calculation, one obtains
\begin{align}
&\int_{z>t_1>\cdots>t_m>0} \frac{(at_m+1)^j-1}{t_1\cdots t_m}dt_m\cdots dt_1\nonumber\\
&=\int_{z>t_1>\cdots>t_{m-1}>0} \frac{1}{t_1\cdots t_{m-1}}\left\{\int_0^{t_{m-1}} \frac{(at_m+1)^j-1}{ t_m}dt_m\right\} dt_{m-1}\cdots dt_1\nonumber\\
&=\sum_{j_1=1}^j \frac1{j_1}\int_{z>t_1>\cdots>t_{m-1}>0} \frac{(at_{m-1}+1)^{j_1}-1}{t_1\cdots t_{m-1}} dt_{m-1}\cdots dt_1=\cdots\nonumber\\
&=\sum_{j_1=1}^j \frac1{j_1} \sum_{j_2=1}^{j_1} \frac1{j_2} \cdots \sum_{j_{m-1}=1}^{j_{m-2}} \frac{1}{j_{m-1}}\sum_{j_m=1}^{j_{m-1}} \frac{(az+1)^{j_m}-1}{j_m}\nonumber\\
&=\sum_{j\geq j_1\geq \cdots\geq j_m\geq 1} \frac{(az+1)^{j_m}-1}{j_1j_2\cdots j_m}\label{equ-MInteMhs}\\
&=\sum_{l=1}^j \frac{(az+1)^{l}-1}{l} \sum_{j\geq j_1\geq \cdots\geq j_{m-1}\geq l} \frac1{j_1\cdots j_{m-1}}\nonumber\\
&=\sum_{l=1}^j \frac{(az+1)^{l}-1}{l} \sum_{j-l+1\geq i_1\geq \cdots\geq i_{m-1}\geq 1} \frac1{(i_1+l-1)\cdots (i_{m-1}+l-1)}.\nonumber
\end{align}
In Lemma \ref{lem-PMPLs2}, replacing $n$ by $j-l+1$ and $l$ by $l-1$, and letting $r=m-1$, $(k_1,\ldots,k_r)=(\{1\}_{m-1})$, we easily obtain
\begin{align}\label{equ-emhn-Bell-Str}
&\sum_{j-l+1\geq i_1\geq \cdots\geq i_{m-1}\geq 1} \frac{1}{(i_1+l-1)\cdots (i_{m-1}+l-1)}\nonumber\\
&\quad\quad=(-1)^{m-1}\sum_{h=0}^{m-1} (-1)^h\ze^\star_j(\{1\}_h)\ze_{l-1}(\{1\}_{m-1-h}).
\end{align}
Applying \eqref{equ-Str-mhs-1}, \eqref{equ-Str-mhss-1} and \eqref{equ-emhn-Bell-Str} yields the desired evaluation.
\end{proof}

\begin{pro}\label{pro-two-miter} For positive integer $j$ and real $\alpha\neq 0$, we have
\begin{align}
\int_0^1 \log^m(1-t)(1+\alpha t)^{j-1}dt=(-1)^{m+1}m! \frac{(\alpha+1)^j}{\alpha j} \sum_{j\geq j_1\geq \cdots\geq j_m\geq1} \frac{(\alpha+1)^{-j_m}-1}{j_1\cdots j_m}.
\end{align}
\end{pro}
\begin{proof}
In \eqref{equ-mi-fs-za}, letting $z=1$ and noting the fact that
\begin{align*}
\int_{1>t_1>\cdots>t_{m-1}>t} \frac{1}{t_1\cdots t_{m-1}}dt_{m-1}\cdots dt_1=\frac{(-1)^{m-1}}{(m-1)!}\log^{m-1}(t),
\end{align*}
the left hand side of \eqref{equ-mi-fs-za} can be rewritten as
\begin{align*}
&\int_{1>t_1>\cdots>t_m>0} \frac{(at_m+1)^j-1}{t_1\cdots t_m}dt_m\cdots dt_1\\
&=\int_0^1 \frac{(at+1)^j-1}{t} \left\{\int_{1>t_1>\cdots>t_{m-1}>t} \frac{1}{t_1\cdots t_{m-1}}dt_{m-1}\cdots dt_1 \right\}dt\\
&=\frac{(-1)^{m-1}}{(m-1)!} \int_0^1 \frac{\log^{m-1}(t)\Big((at+1)^j-1\Big)}{t}dt\\
&=\frac{(-1)^{m}}{m!} aj \int_0^1 \log^m(t)(at+1)^{j-1}dt\quad\quad (t\rightarrow 1-t)\\
&=\frac{(-1)^{m}}{m!} aj \int_0^1 \log^m(1-t)(a(1-t)+1)^{j-1}dt\\
&=\frac{(-1)^{m}}{m!}a(a+1)^{j-1} j \int_0^1 \log^m(1-t)\left(1-\frac{a}{a+1}\right)^{j-1}dt.
\end{align*}
Using \eqref{equ-MInteMhs} with $z=1$ gives
\begin{align*}
&\int_0^1 \log^m(1-t)\left(1-\frac{a}{a+1}t\right)^{j-1}dt\\
&=(-1)^m m!\frac{1}{a(a+1)^{j-1}j} \sum_{j\geq j_1\geq \cdots\geq j_m\geq 1} \frac{(a+1)^{j_m}-1}{j_1j_2\cdots j_m}.
\end{align*}
Finally, setting $\alpha=-a/(a+1)\ (a\neq0,-1)$, by an elementary calculation, we obtain the desired evaluation.
\end{proof}

\subsection{Proof of Theorems \ref{thm:Mneimneh-type-mhtss}, \ref{thm:Mneimneh-type-mhtss-2} and \ref{thm:Mneimneh-type-mhtss-3}}

From Lemmas \ref{lem-Bell-Log} and \ref{lem-x-mhs-Log}, by a simple calculation, we arrive at the result
\begin{align}
\ze^\star_n(\{1\}_p;z)-\ze^\star_n(\{1\}_p)=\frac{n}{p!} \sum_{j=0}^p (-1)^j \binom{p}{j} \log^{p-j}(1-z) \int_1^z t^{n-1} \log^j(1-t)dt.
\end{align}
Hence,
\begin{align*}
&\sum_{k=0}^n x^ky^{n-k}\binom{n}{k} \left(\zeta_k^\star(\{1\}_p;z)-\zeta_k^\star(\{1\}_p\right)\\
&=\sum_{k=0}^n x^ky^{n-k}\binom{n}{k} k \sum_{j=0}^p \frac{(-1)^j}{j!} \frac{\log^{p-j}(1-z)}{(p-j)!} \int_1^z t^{k-1} \log^j(1-t)dt\\
&=nx \sum_{j=0}^p \frac{(-1)^j}{j!} \frac{\log^{p-j}(1-z)}{(p-j)!}  \int_1^z \log^j(1-t) (xt+y)^{n-1}dt \quad\quad\quad(\text{letting}\ u=tx+y)\\
&=n \sum_{j=0}^p \frac{(-1)^j}{j!} \frac{\log^{p-j}(1-z)}{(p-j)!}  \int_{x+y}^{xz+y} u^{n-1}\log^j\Big(\frac{x+y-u}{x}\Big) du \\
&=n \sum_{j=0}^p \frac{(-1)^j}{j!} \frac{\log^{p-j}(1-z)}{(p-j)!}  \int_{x+y}^{xz+y} u^{n-1} \left\{\log\Big(\frac{x+y}{x}\Big)+\log\Big(1-\frac{u}{x+y}\Big) \right\}^j du \\
&=n \sum_{j=0}^p \frac{(-1)^j}{j!} \frac{\log^{p-j}(1-z)}{(p-j)!} \sum_{i=0}^j \binom{j}{i} \log^{j-i} \Big(\frac{x+y}{x}\Big) \int_{x+y}^{xz+y} u^{n-1}\log^i\Big(1-\frac{u}{x+y}\Big)du\\
&=n(x+y)^n \sum_{j=0}^p \frac{(-1)^j}{j!} \frac{\log^{p-j}(1-z)}{(p-j)!} \sum_{i=0}^j \binom{j}{i} \log^{j-i} \Big(\frac{x+y}{x}\Big) \int_{1}^{(xz+y)(x+y)^{-1}} v^{n-1}\log^i(1-v)dv.
\end{align*}
Applying \eqref{equ-Bell-Log} and \eqref{equ-log-x-m} gives
\begin{align*}
&\sum_{k=0}^n x^ky^{n-k}\binom{n}{k} \left(\zeta_k^\star(\{1\}_p;z)-\zeta_k^\star(\{1\}_p\right)\\
&=(x+y)^n \sum_{j=0}^p \frac{(-1)^j}{j!} \frac{\log^{p-j}(1-z)}{(p-j)!} \sum_{i=0}^j \binom{j}{i} \log^{j-i} \Big(\frac{x+y}{x}\Big) \\&\quad\quad\quad\quad\quad\times\sum_{l=0}^i (-1)^l l! \binom{i}{l} \log^{i-l}\Big(\frac{(1-z)x}{x+y}\Big) \left(\ze^\star_n\Big(\{1\}_l;\frac{xz+y}{x+y}\Big)-\ze^\star_n(\{1\}_l)\right)\\
&=(x+y)^n \sum_{p\geq j\geq i\geq l\geq 0} (-1)^{j+l} \frac{\log^{p-j}(1-z)\log^{j-i} \Big(\frac{x+y}{x}\Big) \log^{i-l}\Big(\frac{(1-z)x}{x+y}\Big)}{(p-j)!(j-i)!(i-l)!} \\&\quad\quad\quad\quad\quad\quad\quad\quad\quad\times\left(\ze^\star_n\Big(\{1\}_l;\frac{xz+y}{x+y}\Big)-\ze^\star_n(\{1\}_l)\right)\\
&=(x+y)^n \sum_{l+h+g+q=p,\atop l,h,g,q\geq 0} (-1)^{h+g} \frac{\log^{q}(1-z)\log^{g} \Big(\frac{x+y}{x}\Big) \log^{h}\Big(\frac{(1-z)x}{x+y}\Big)}{q!g!h!}
\\&\quad\quad\quad\quad\quad\quad\quad\quad\quad\times\left(\ze^\star_n\Big(\{1\}_l;\frac{xz+y}{x+y}\Big)-\ze^\star_n(\{1\}_l)\right)\\
&=(x+y)^n \sum_{l=0}^p \left(\ze^\star_n\Big(\{1\}_l;\frac{xz+y}{x+y}\Big)-\ze^\star_n(\{1\}_l)\right) \\&\quad\quad\quad\quad\quad\quad\quad\times\sum_{h+g+q=p-l,\atop h,g,q\geq 0} (-1)^{h+g} \frac{\log^{q}(1-z)\log^{g} \Big(\frac{x+y}{x}\Big) \log^{h}\Big(\frac{(1-z)x}{x+y}\Big)}{q!g!h!}\\
&=(x+y)^n \sum_{l=0}^p \left(\ze^\star_n\Big(\{1\}_l;\frac{xz+y}{x+y}\Big)-\ze^\star_n(\{1\}_l)\right) \\&\quad\quad\quad\quad\quad\quad\quad\times\frac{1}{(p-l)!} \left\{\log(1-z)+\log\Big(\frac{x}{x+y}\Big)+\log\Big(\frac{x+y}{(1-z)x}\Big)\right\}^{p-l}\\
&=(x+y)^n \left(\ze^\star_n\Big(\{1\}_p;\frac{xz+y}{x+y}\Big)-\ze^\star_n(\{1\}_p)\right).
\end{align*}
Clearly, if setting $z=0$ then
\begin{align*}
&\sum_{k=0}^n x^ky^{n-k}\binom{n}{k} \zeta_k^\star(\{1\}_p)=(x+y)^n \left(\ze^\star_n(\{1\}_p)-\ze^\star_n\Big(\{1\}_p;\frac{y}{x+y}\Big)\right).
\end{align*}
Therefore, we have
\begin{align}\label{equ-mhn-1-z-mhss-2}
&\sum_{k=0}^n x^ky^{n-k}\binom{n}{k} \zeta_k^\star(\{1\}_p;z)\nonumber\\&=(x+y)^n \left(\ze^\star_n\Big(\{1\}_p;\frac{xz+y}{x+y}\Big)-\ze^\star_n\Big(\{1\}_p;\frac{y}{x+y}\Big)\right)\nonumber\\
&=(x+y)^n \sum_{n\geq n_1\geq \cdots\geq n_p\geq 1 } \frac{\Big(\frac{xz+y}{x+y}\Big)^{n_p}-\Big(\frac{y}{x+y}\Big)^{n_p}}{n_1\cdots n_p}\\
&=(x+y)^n \sum_{j=1}^n \left(\Big(\frac{xz+y}{x+y}\Big)^{j}-\Big(\frac{y}{x+y}\Big)^{j} \right) \sum_{n\geq n_1\geq \cdots\geq n_{p-1}\geq j } \frac{1}{n_1\cdots n_{p-1}}.\nonumber
\end{align}
Finally, applying \eqref{equ-Str-mhs-1}, \eqref{equ-Str-mhss-1} and \eqref{equ-emhn-Bell-Str} yields the desired evaluation. Thus, this concludes
the proof of Theorem \ref{thm:Mneimneh-type-mhtss}.

Replacing $z$ by $t_m$ in \eqref{equ:Mneimneh-type-mhtss}, then applying $\int_{z>t_1>\cdots>t_m>0} \frac{(\cdot)}{t_1\cdots t_m}dt_m\cdots dt_1$ to it and using Proposition \ref{pro-one-miter}, we deduce
\begin{align*}
&\sum_{k=0}^n x^ky^{n-k}\binom{n}{k} \zeta_k^\star(\{1\}_{p-1},m+1;z)\nonumber\\&=(-1)^{p-1}(x+y)^n \sum_{j=1}^n \left\{\sum_{i=0}^{p-1} (-1)^i \frac{Y_i(n)}{i!} \frac{s(j,p-i)}{j!}\right\} (x+y)^{-j}y^{j} \\&\quad\quad\quad\quad\quad\quad\quad\times \int_{z>t_1>\cdots>t_m>0} \frac{\left(\frac{x}{y}t_m+1\right)^j-1}{t_1\cdots t_m}dt_m\cdots dt_1\\
&=(-1)^{p+m}(x+y)^n \sum_{j=1}^n \left\{\sum_{i=0}^{p-1} (-1)^i \frac{Y_i(n)}{i!} \frac{s(j,p-i)}{j!}\right\} (x+y)^{-j}y^{j} \\&\quad\quad\quad\quad\quad\quad\quad\times\sum_{l=1}^j \left(\Big(\frac{xz}{y}+1\Big)^l-1\right)\left\{\sum_{h=0}^{m-1} (-1)^h \frac{Y_h(j)}{h!} \frac{s(l,m-h)}{l!}\right\}.
\end{align*}
Thus, this completes the proof of Theorem \ref{thm:Mneimneh-type-mhtss-2}.

On the other hand, applying $\int_{z>t_1>\cdots>t_m>0} \frac{(\cdot)}{t_1\cdots t_m}dt_m\cdots dt_1$ to \eqref{equ-mhn-1-z-mhss-2} and using \eqref{equ-MInteMhs}, we obtain another form of \eqref{equ:Mneimneh-type-mhtss-2} ($p\in \N,\ m\in \N_0$)
\begin{align}\label{equ:Mneimneh-type-mhtss-2-anh}
&\sum_{k=0}^n x^ky^{n-k}\binom{n}{k} \zeta_k^\star(\{1\}_{p-1},m+1;z)\nonumber\\&
=(x+y)^n \sum_{n\geq n_1\geq \cdots \geq n_{p+m}\geq 1} \frac{(x+y)^{-n_p}y^{n_p-n_{p+m}}\left((xz+y)^{n_{p+m}}-y^{n_{p+m}}\right)}{n_1\cdots n_{p+m}}.
\end{align}
Multiplying \eqref{equ:Mneimneh-type-mhtss-2-anh} by $\log^{r-1}(1-z)/z$ and integrating over the interval (0,1), and applying \eqref{equ-Bell-Log}, we have
\begin{align}\label{equ:Mneimneh-type-mhtss-2-anh}
&(-1)^{r-1}(r-1)!\sum_{k=0}^n x^ky^{n-k}\binom{n}{k} \zeta_k^\star(\{1\}_{p-1},m+1;z)\nonumber\\&
=(x+y)^n \sum_{n\geq n_1\geq \cdots \geq n_{p+m}\geq 1} \frac{(x+y)^{-n_p}y^{n_p}}{n_1\cdots n_{p+m}}\int_0^1 \frac{\log^{r-1}(1-z)\left(\Big(\frac{xz}{y}+1\Big)^{n_{p+m}}-1\right)}{z}dz\nonumber\\
&=(x+y)^n \sum_{n\geq n_1\geq \cdots \geq n_{p+m}\geq 1} \frac{(x+y)^{-n_p}y^{n_p}}{n_1\cdots n_{p+m}} \sum_{n_{p+m+1}=1}^{n_{p+m}} \frac{x}{y} \int_0^1 \log^{r-1}(1-z)\Big(\frac{xz}{y}+1\Big)^{n_{p+m+1}-1}dz.
\end{align}
Finally, using Proposition \ref{pro-two-miter}, by a straight-forward calculation, we obtain \eqref{equ:Mneimneh-type-mhtss-3}. Thus, this concludes the proof of the Theorem \ref{thm:Mneimneh-type-mhtss-3}.

\subsection{Corollaries and Examples}

From Theorem \ref{thm:Mneimneh-type-mhtss-2}, letting $m=1$ or $p=1$ give the following corollaries.
\begin{cor}\label{cor:Mneimneh-type-mhtss-2-cor-1} For any reals $x,y$ and $z\in (-\infty,1]$ with $x/(x+y)\geq 0$ and $n,p\in \N$, we have
\begin{align}\label{equ:Mneimneh-type-mhtss-2-cor-1}
&\sum_{k=0}^n x^ky^{n-k}\binom{n}{k} \zeta_k^\star(\{1\}_{p-1},2;z)\nonumber\\&=(-1)^{p-1}(x+y)^n \sum_{n\geq j\geq l\geq 1} \frac1{l}\left(\frac{y}{x+y}\right)^j \left( \left(1+\frac{xz}{y}\right)^l-1\right)\left\{\sum_{i=0}^{p-1} (-1)^i \frac{Y_i(n)}{i!} \frac{s(j,p-i)}{j!}\right\}.
\end{align}
\end{cor}

\begin{cor}\label{cor:Mneimneh-type-mhtss-2-cor-two} For any reals $x,y$ and $z\in (-\infty,1]$ with $x/(x+y)\geq 0$ and $n,m\in \N$, we have
\begin{align}\label{equ:Mneimneh-type-mhtss-2-cor-2}
&\sum_{k=0}^n x^ky^{n-k}\binom{n}{k} \zeta_k^\star(m+1;z)\nonumber\\&=(-1)^{m-1}(x+y)^n \sum_{n\geq j\geq l\geq 1} \frac1{j}\left(\frac{y}{x+y}\right)^j \left( \left(1+\frac{xz}{y}\right)^l-1\right)\left\{\sum_{h=0}^{m-1} (-1)^h \frac{Y_h(j)}{h!} \frac{s(l,m-h)}{l!}\right\}.
\end{align}
\end{cor}

Further, from Theorems \ref{thm:Mneimneh-type-mhtss}-\ref{thm:Mneimneh-type-mhtss-3}, and Corollaries \ref{cor:Mneimneh-type-mhtss-2-cor-1}-\ref{cor:Mneimneh-type-mhtss-2-cor-two}, we obtain the following specific cases.
\begin{exa} For any reals $x,y$ with $x/(x+y)\geq 0$,
\begin{align*}
&\sum_{k=0}^n x^ky^{n-k}\binom{n}{k} {\bar H}^{(2)}_k=\sum_{n\geq j\geq l\geq 1} \frac{1}{jl}(x+y)^{n-j}y^{j-l}\Big(y^l-(y-x)^l\Big),\\
&\sum_{k=0}^n x^ky^{n-k}\binom{n}{k} {H}^{(3)}_k=\sum_{n\geq j\geq l\geq 1} \frac{1}{jl}(x+y)^{n-j}y^{j-l}\Big((x+y)^l-y^l\Big)(H_j-H_{l-1}),\\
&\sum_{k=0}^n x^ky^{n-k}\binom{n}{k} \ze^\star_k(1,2)=\sum_{n\geq j\geq l\geq 1} \frac{1}{jl}(x+y)^{n-j}y^{j-l}\Big((x+y)^l-y^l\Big)(H_n-H_{j-1}),\\
&\sum_{k=0}^n x^ky^{n-k}\binom{n}{k} \ze^\star_k(2,1)=\sum_{n\geq j\geq l\geq h\geq 1} \frac{1}{jlh} (x+y)^{n-j}y^{j-l} (x+y)^l \Big(1-(x+y)^{-h}y^h\Big),
\end{align*}
where ${\bar H}^{(2)}_k$ is the \emph{$k$-th alternating harmonic number of order 2} defined by
\begin{align*}
{\bar H}^{(2)}_k:=\sum_{j=1}^k \frac{(-1)^{j-1}}{j^2}.
\end{align*}
More generally, the \emph{$k$-th alternating harmonic number of order p} ${\bar H}^{(p)}_k\ (p\in \N)$ is defined by
\begin{align*}
{\bar H}^{(p)}_k:=\sum_{j=1}^k \frac{(-1)^{j-1}}{j^p}\quad \text{and} \quad {\bar H}^{(p)}_0:=0.
\end{align*}
\end{exa}

Noting the fact that
\[H_kH_k^{(2)}=\ze^\star_k(1,2)+\ze^\star_k(2,1)-H_k^{(3)},\]
we get the Mneimneh-type binomial sums involving two harmonic numbers
\begin{align}\label{equ-harmonic-12}
&\sum_{k=0}^n x^ky^{n-k}\binom{n}{k}H_kH_k^{(2)}\nonumber \\
&=\sum_{k=0}^n x^ky^{n-k}\binom{n}{k} \ze^\star_k(1,2)+\sum_{k=0}^n x^ky^{n-k}\binom{n}{k} \ze^\star_k(2,1)-\sum_{k=0}^n x^ky^{n-k}\binom{n}{k} {H}^{(3)}_k\nonumber \\
&=\sum_{n\geq j\geq l\geq 1} \frac{1}{jl}(x+y)^{n-j}y^{j-l}\Big((x+y)^l-y^l\Big)\Big(H_n-2H_{j}+H_l+\frac1{j}-\frac1{l}\Big)\nonumber \\
&\quad+\sum_{n\geq j\geq l\geq 1} \frac{1}{jl}(x+y)^{n-j}y^{j-l}(x+y)^l \left(H_l-\sum_{h=1}^l \frac{(x+y)^{-h}y^h}{h}\right).
\end{align}
On the other hand, from \cite[Thm. 1.1]{PX2024}, setting $p=$3 yields
\begin{align}\label{equ-harmonic-123}
&\sum_{k=0}^n x^ky^{n-k}\binom{n}{k}\Big(H_k^3+3H_kH_k^{(2)}+2H_k^{(3)}\Big)\nonumber \\
&=3(x+y)^n \sum_{j=0}^n \frac{1}{j}\left(1-\Big(\frac{y}{x+y}\Big)^j\right)\left\{H_n^2+H_n^{(2)}-2H_nH_{j-1}+H_{j-1}^2-H_{j-1}^{(2)} \right\}.
\end{align}
Hence, by an elementary calculation, we deduce the Mneimneh-type binomial sums involving cubic harmonic numbers
\begin{align}\label{equ-harmonic-123}
&\sum_{k=0}^n x^ky^{n-k}\binom{n}{k}H_k^3\nonumber \\
&=3(x+y)^n \sum_{j=0}^n \frac{1}{j}\left(1-\Big(\frac{y}{x+y}\Big)^j\right)\left\{H_n^2+H_n^{(2)}-2H_nH_{j-1}+H_{j-1}^2-H_{j-1}^{(2)} \right\}\nonumber\\
&\quad-\sum_{n\geq j\geq l\geq 1} \frac{1}{jl}(x+y)^{n-j}y^{j-l}\Big((x+y)^l-y^l\Big)\Big(3H_n-4H_{j}+H_l+\frac3{j}-\frac1{l}\Big)\nonumber \\
&\quad-3\sum_{n\geq j\geq l\geq 1} \frac{1}{jl}(x+y)^{n-j}y^{j-l}(x+y)^l \left(H_l-\sum_{h=1}^l \frac{(x+y)^{-h}y^h}{h}\right).
\end{align}

\subsection{Proof of Theorem \ref{thm:Mneimneh-type-mhtss-4}}
The proof of Theorem \ref{thm:Mneimneh-type-mhtss-4} is based on the following well-known Toeplitz limit theorem.

\begin{lem}\label{lem-Kn1990} (\cite[Thm. 4]{Kn1990}) Assume that $\lim_{n\rightarrow \infty }x_n=0$ and the terms $a_{nk}$ form a triangular system
\begin{align*}
\left( {\begin{array}{*{20}{c}}
{{a_{00}}}&0&0& \cdots &0\\
{{a_{10}}}&{{a_{11}}}&0& \cdots &0\\
{{a_{20}}}&{{a_{21}}}&{{a_{22}}}& \cdots &0\\
 \vdots & \vdots & \vdots & \ddots & \vdots \\
{{a_{n0}}}&{{a_{n1}}}&{{a_{n2}}}& \cdots &{{a_{nn}}}
\end{array}} \right)
\end{align*}
and satisfy the following conditions:
(a)  every column of the system contains a null sequence, i.e. for fixed $k\geq 0$, $\lim_{n\rightarrow \infty} a_{nk}=0$, (b) there exists a constant $K$ independent of $n$ such that $\sum_{k=0}^n |a_{nk}|<K$. Then
\begin{align}
\lim_{n\rightarrow \infty} \sum_{k=0}^n a_{nk} x_k=0.
\end{align}
\end{lem}
{\bf Proof of Theorem \ref{thm:Mneimneh-type-mhtss-4}}. We only prove that Theorem \ref{thm:Mneimneh-type-mhtss-4} holds for $r\geq 2$. From Theorem \ref{thm:Mneimneh-type-mhtss-3}, setting $p=1,\ x+y=1$ and $x\in [0,1)$ gives
\begin{align}\label{equ-special-casep=1}
\sum_{k=0}^n (1-y)^ky^{n-k}\binom{n}{k} \zeta_k^\star(m+2,\{1\}_{r-1})= \sum_{n\geq n_1\geq \cdots\geq n_{m+r+1}\geq 1} \frac{y^{n_1-n_{m+2}}(1-y^{n_{m+r+1}})}{n_1\cdots n_{m+r+1}}.
\end{align}
By the relation in \eqref{equ-special-casep=1}, letting $n\rightarrow \infty$, we immediately discover that for $y\neq 0$
\begin{align}\label{equ-special-casep=1-one}
&\lim_{n\rightarrow \infty} \sum_{k=0}^n (1-y)^ky^{n-k}\binom{n}{k} \zeta_k^\star(m+2,\{1\}_{r-1})\nonumber\\
&=\Li^\star_{\{1\}_{m+r+1}}\big(y,\{1\}_m,y^{-1},\{1\}_{r-1}\big)-\Li^\star_{\{1\}_{m+r+1}}\big(y,\{1\}_m,y^{-1},\{1\}_{r-2},y\big).
\end{align}
Therefore, for the proof of \eqref{equ:Mneimneh-type-mhtss-4}, it suffices to show that
\begin{align}\label{equ-special-casep=1-two}
\lim_{n\rightarrow \infty} \left( \zeta^\star(m+2,\{1\}_{r-1})-\sum_{k=0}^n (1-y)^ky^{n-k}\binom{n}{k} \zeta_k^\star(m+2,\{1\}_{r-1})\right)=0,
\end{align}
where $m,r,n$ and $y$ satisfy the assumptions of Theorem \ref{thm:Mneimneh-type-mhtss-4}.

Noting the fact that
\[\sum_{k=0}^n (1-y)^ky^{n-k}\binom{n}{k}=1,\]
we have
\begin{align}\label{equ-special-casep=1-three}
&\zeta^\star(m+2,\{1\}_{r-1})-\sum_{k=0}^n (1-y)^ky^{n-k}\binom{n}{k} \zeta_k^\star(m+2,\{1\}_{r-1})\nonumber\\
&=\sum_{k=0}^n (1-y)^ky^{n-k}\binom{n}{k} \left(\zeta^\star(m+2,\{1\}_{r-1})-\zeta_k^\star(m+2,\{1\}_{r-1})\right).
\end{align}
Now, set for every $k,n\in \N$,
\begin{align*}
&a_{nk}:=\binom{n}{k} (1-y)^k y^{n-k}\quad\text{and}\quad x_k:=\zeta^\star(m+2,\{1\}_{r-1})-\zeta_k^\star(m+2,\{1\}_{r-1})
\end{align*}
and let us verify the assumptions of Lemma \ref{lem-Kn1990}. Clearly,  the condition $\lim_{k\rightarrow \infty}x_k=0$ and condition (a) $\lim_{n\rightarrow \infty }a_{nk}=0$ are satisfied (since $y\in (0,1)$). To fulfill condition (b), it is necessary that $\sum_{k=0}^n |a_{nk}|$ is bounded. Noticing that
\begin{align}\label{equ-special-casep=1-four}
\sum_{k=0}^n |a_{nk}|=\sum_{k=0}^n \binom{n}{k} |(1-y)^ky^{n-k}|=(1-y+y)^n=1,
\end{align}
and applying Lemma 3.9, we obtain the desired evaluation \eqref{equ:Mneimneh-type-mhtss-4}. Similarly, we also prove \eqref{equ:Mneimneh-type-mhtss-4-2}. This concludes the proof of Theorem \ref{thm:Mneimneh-type-mhtss-4}.

\section{A Conjecture}\label{sec-conj}

For any ${\bfk}_{}:=(k_1,k_2,\ldots,\ k_{r})$, we put
\begin{align*}
|{\bfk}|_j:=k_1+k_2+\cdots+k_j\quad\text{and}\quad |{\bfk}|_0:=0.
\end{align*}

Now, we end this paper by the following conjecture.
\begin{con}\label{con:Mneimneh-type-mhtss} Let $r\in \N_0$ and $n\in \N$. For any reals $x,y$ and ${\bf p}_{r+1}:=(p_1,p_2,\ldots,\ p_{r+1})\in \N^{r+1},\ {\bf m}_{r}:=(m_1,\ldots,m_r)\in \N_0^r$, we have
\begin{align}\label{conequ:Mneimneh-type-mhtss}
&\sum_{k=0}^n x^ky^{n-k}\binom{n}{k} \zeta_k^\star\left(\{1\}_{p_1-1},m_1+2,\ldots,\{1\}_{p_r-1},m_r+2,\{1\}_{p_{r+1}-1}\right)\nonumber\\&
=(x+y)^n \sum_{n\geq n_1\geq \cdots \geq n_{|{\bf p}|_{r+1}+|{\bf m}|_r+r-1}\geq 1} \frac{\left(\frac{y}{x+y}\right)^{\sum_{j=1}^r \Big(n_{|{\bf p}|_{j}+|{\bf m}|_{j-1}+j-1}-n_{|{\bf p}|_{j}+|{\bf m}|_{j}+j}\Big)}}{n_1\cdots n_{|{\bf p}|_{r+1}+|{\bf m}|_r+r-1}}\nonumber\\&
\quad\quad\quad\quad\quad\quad\quad\quad\quad\quad\quad\quad\quad\quad\quad\quad\times \left(1-\left(\frac{y}{x+y}\right)^{n_{|{\bf p}|_{r+1}+|{\bf m}|_r+r-1}}\right).
\end{align}
\end{con}

\begin{re}
Applying the Toeplitz principle, we also deduce the corresponding (conjectural) functional multiple polylogarithmic identity similar to Theorem \ref{thm:Mneimneh-type-mhtss-4}.
\end{re}

\medskip

{\bf Declaration of competing interest.}
The authors declares that they has no known competing financial interests or personal relationships that could have
appeared to influence the work reported in this paper.

{\bf Data availability.}
No data was used for the research described in the article.

{\bf Acknowledgments.} The authors are grateful to Prof. M. Gen$\check{\rm c}$ev for his opinion concerning the relation in Theorem \ref{thm:Mneimneh-type-mhtss-4}. Ce Xu is supported by the National Natural Science Foundation of China (Grant No. 12101008), the Natural Science Foundation of Anhui Province (Grant No. 2108085QA01) and the University Natural Science Research Project of Anhui Province (Grant No. KJ2020A0057).

\end{document}